\documentclass[11pt]{article}
\usepackage{mathrsfs}
\usepackage{amsmath,amsmath,amsthm,amssymb}
\usepackage{fancyhdr,graphicx}
\usepackage{enumerate}
\usepackage{hyperref}

\numberwithin{equation}{section}

\newtheorem{theorem}{Theorem}

\newtheorem {corollary}{Corollary}
\newtheorem {lemma}{Lemma}
\newtheorem {remark}{Remark}
\newtheorem*{definition}{Definition}

\newcommand{\ds}{\displaystyle}

\begin{document}
%\setlength{\parindent}{4ex} \setlength{\parskip}{1ex}
%\setlength{\oddsidemargin}{12mm} \setlength{\evensidemargin}{9mm}
%%---------------------------------------------------------------------------------------------------------
%                                          Title
%%---------------------------------------------------------------------------------------------------------
\title{Shooting with degree theory: Analysis of some weighted poly-harmonic systems}
\author{John Villavert\\
[0.2cm] {\small Department of Mathematics, University of Oklahoma}\\
{\small Norman, OK 73019, USA} 
}
\date{}

\maketitle
%% ------------------------------------------------------------------------------------------------------------------------------
%                                      Abstract
%%-------------------------------------------------------------------------------------------------------------------------------
\begin{abstract}
In this paper, the author establishes the existence of positive entire solutions to a general class of semilinear poly-harmonic systems, which includes equations and systems of the weighted Hardy--Littlewood--Sobolev type. The novel method used implements the classical shooting method enhanced by topological degree theory. The key steps of the method are to first construct a target map which aims the shooting method and the non-degeneracy conditions guarantee the continuity of this map. With the continuity of the target map, a topological argument is used to show the existence of zeros of the target map. The existence of zeros of the map along with a non-existence theorem for the corresponding Navier boundary value problem imply the existence of positive solutions for the class of poly-harmonic systems.
\end{abstract}
 {\small \noindent{\bf Keywords:}\, Degree theory; Lane--Emden system; Poly-harmonic systems; The shooting method;  Weighted Hardy--Littlewood--Sobolev inequality.\\
\noindent{\bf Mathematics Subject Classification: \,}  35B09; 35B33; 35J30; 35J48; 47H11.
\maketitle
%%-------------------------------------------------------------------------------------------------------------------------------
%                                     Introduction
%%-------------------------------------------------------------------------------------------------------------------------------

\section{Introduction}
This manuscript establishes the existence of positive solutions and related properties for higher-order, nonlinear system of elliptic equations in the whole space. As we shall see, the class of problems we examine includes the well-known Lane--Emden and Hardy--Littlewood--Sobolev type systems along with their weighted counterparts as motivating examples. The general framework we adopt to establish these existence results is inspired by the recent work of Li in \cite{Li11}. Remarkably, the mathematical tools utilized within this framework are more or less elementary by themselves, but we combine them together to obtain some new and interesting results. Our first main result proves the existence of positive entire solutions, under reasonable assumptions, to the general system
\begin{equation}\label{general}
(-\Delta)^{k_{i}}u_{i} = f_{i}(|x|,u_{1},u_{2},\ldots,u_{L}) \, \text{ in } \mathbb{R}^{n}\backslash\{0\},\,\,\,\, \text{ for } i = 1,2,\ldots, L.
\end{equation}
As we demonstrate below, proving the existence of positive solutions to this system of poly-harmonic equations involves reformulating the problem in radial coordinates then applying the classical shooting method combined with a non-existence theorem for the corresponding Navier boundary value problem. Specifically, a natural ingredient of the proof entails constructing a continuous {\bf target map} which aims the shooting method. Then a topological argument via degree theory is invoked to guarantee the existence of zeros of this target map, which enables us to identify the correct initial shooting positions for the shooting method. By combining this with a non-existence result for the corresponding boundary value problem, we obtain the existence of positive (radial) solutions to system \eqref{general}. 

In establishing our general results, the primary examples we consider are the weighted Hardy--Littlewood--Sobolev (HLS) equation 
\begin{align}\label{wHLS equationp}
  \left\{\begin{array}{ccl}
    (-\Delta)^{\gamma/ 2} u = \displaystyle\frac{u^{p}}{|x|^{\sigma}} & \text{ in } & \mathbb{R}^n\backslash\{0\}, \\
    u > 0 & \text{ in } & \mathbb{R}^{n},
  \end{array}
\right.
\end{align}
and the weighted HLS system,
\begin{align}\label{wHLS systemp}
  \left\{\begin{array}{ccl}
    (-\Delta)^{\gamma/ 2} u = \displaystyle\frac{v^{q}}{|x|^{\sigma_1}} & \text{ in } &\mathbb{R}^n\backslash\{0\}, \\
    (-\Delta)^{\gamma/ 2} v = \displaystyle\frac{u^{p}}{|x|^{\sigma_2}} & \text{ in } &\mathbb{R}^n\backslash\{0\}, \\
    u,~v > 0 & \text{ in } &\mathbb{R}^{n}.
  \end{array}
\right.
\end{align}
Here, $n\geq 3$, $\gamma \in (0,n)$, $\sigma,\sigma_{1},\sigma_{2} \in \mathbb{R}$ and $p$ and $q$ are positive exponents. Notice that when $\gamma = 2$ and $\sigma_1 = \sigma_2 = 0$, the weighted system reduces to the well-known Lane--Emden system
\begin{align}\label{Lane-Emden}
  \left\{\begin{array}{ccc}
    -\Delta u = v^{q}, & u > 0, & \text{ in } \mathbb{R}^n, \\
    -\Delta v = u^{p}, & v > 0, & \text{ in } \mathbb{R}^n,
  \end{array}
\right.
\end{align}
or more generally to the HLS system when $\gamma > 2$ and $\sigma_1 = \sigma_2 = 0$:
\begin{align}\label{HLS PDE}
  \left\{\begin{array}{ccc}
    \ds(-\Delta)^{\gamma/ 2} u = v^{q}, & u > 0, & 	\text{ in } \mathbb{R}^n, \\
    \ds(-\Delta)^{\gamma/ 2} v = u^{p}, & v > 0, & 	\text{ in } \mathbb{R}^n.
  \end{array}
\right.
\end{align}
The Lane--Emden and HLS systems have received much attention in the past few decades. For instance, the scalar case was studied in \cite{BN83,CGS89,GS81}, and similar problems have been approached geometrically including the prescribing Gaussian and scalar curvature problems (cf. \cite{CY97,CL95,CL97}). Related systems including its generalized version, the HLS type systems, have been studied as well (cf. \cite{CL91}--\cite{CL09a}, \cite{CLO05a}--\cite{CLO06}, \cite{LL11}--\cite{LLM11}, \cite{LV11,NS86,Phan2012,PhanSouplet2012,PS86,SZ02} and the references therein). %\cite{CL91,CL95,CL95a,CL97,CL01,CL08,CL09a,CLO05a,CLO06,LL11,LLM11a,LLM11,LV11,NS86,PS86,SZ02}. 
When $\gamma$ is an even integer, \eqref{HLS PDE} is equivalent to the integral system
\begin{equation}\label{IEs}
\left \{
\begin{array}{l}
      u(x) = \ds\int_{\mathbb{R}^{n}} \frac{v(y)^q}{|x-y|^{n-\gamma}}  dy,  \,\,
      u > 0 \,\, \mbox{ in } \mathbb{R}^n,\\
      v(x) = \ds\int_{\mathbb{R}^{n}} \frac{u(y)^p}{|x-y|^{n-\gamma}}dy, \,\,
      v > 0  \,\, \mbox{ in } \mathbb{R}^n, 
\end{array}
\right. 
\end{equation}
in the sense that a solution of one system, multiplied by a suitable constant if necessary, is also a solution of the other when $p, q > 1$, and vice versa. Hence, the PDE system \eqref{HLS PDE} and the integral system \eqref{IEs} are both referred to as the HLS system. Now, when studying the HLS system, the exponents $p,q$, and the order $\gamma$ play an essential role in determining the criteria for the existence and non-existence of solutions. More precisely, there are three important cases to consider: The HLS system is said to be in the {\bf subcritical} case if $\frac{1}{1 + p} + \frac{1}{1 + q} > \frac{n-\gamma}{n}$, in the {\bf critical} case if $\frac{1}{1 + p} + \frac{1}{1 + q} = \frac{n-\gamma}{n}$, and in the {\bf supercritical} case if $\frac{1}{1 + p} + \frac{1}{1 + q} < \frac{n-\gamma}{n}$. In the special case of \eqref{Lane-Emden}, the famous Lane--Emden conjecture---an analogue to the celebrated result of Gidas and Spruck in \cite{GS81} for the scalar case---states that this elliptic system in the subcritical case has no classical solution. This has been completely settled for radial solutions (cf. \cite{Mitidieri96,SZ94}), for dimensions $n\leq 4$ (cf. \cite{PQS07,SZ96,Souplet09}), and for $n\geq 5$ but under certain subregions of subcritical exponents (cf. \cite{BM02,DeFF94,Mitidieri96,RZ00,Souplet09,Souto95}). With the help of the method of moving planes in integral form, the work in \cite{CL09}---when combined with the non-existence results in \cite{Mitidieri93}---provides a partial resolution of this conjecture as well. On the other hand, it is interesting to note that the results in this paper also include the existence of solutions to the Lane--Emden system in the non-subcritical case.

Let us further motivate the importance of the HLS system and its related systems in connection with the study of the classical HLS inequality. Recall the HLS inequality states that
\begin{equation}\label{HLSin}
    \ds\int_{\mathbb{R}^n}\int_{\mathbb{R}^n} \frac{f(x) g(y)}{|x - y|^{\lambda}} dx dy \leq
    C_{s,\lambda, n} \|f\|_r\|g\|_s
\end{equation} where $0
<\lambda < n$, $1 < s, r < \infty$, $\frac{1}{r} +
\frac{1}{s} + \frac{\lambda}{n} = 2$, $f \in L^r(\mathbb{R}^n)$, and $g \in L^s(\mathbb{R}^n)$ (cf. \cite{HL,Lieb83,Sobolev63}). To find the best constant in the HLS inequality, one maximizes the associated HLS functional
\begin{equation}\label{J} 
J(f,g) = \ds\int_{\mathbb{R}^n}\int_{\mathbb{R}^n} \frac{f(x)g(y)}{|x-y|^{\lambda}} dx dy
\end{equation}
under the constraint $ \|f\|_r = \|g\|_s =1 $. Let $p=\frac{1}{r-1}$,
$q=\frac{1}{s-1}$ and with a suitable scaling such as $u= c_1 f^{r-1}$ and $v=c_2 g^{s-1}$, the
Euler--Lagrange equations are precisely the system of integral equations in \eqref{IEs}. Here, $u \in L^{p+1}(\mathbb{R}^n)$ and $v \in L^{q+1}(\mathbb{R}^n)$ where the positive exponents $p$ and $q$ are in the critical case. In \cite{Lieb83}, Lieb proved the existence of positive solutions to (\ref{IEs}) which maximize the corresponding functionals $J(f,g)$ in the class of $u \in L^{p+1}(\mathbb{R}^n)$ and $v \in L^{q+1}(\mathbb{R}^n)$. In other words, there exist extremal functions of \eqref{J}, thereby proving the existence of ground state solutions to the HLS system in the critical case. In addition, Hardy and Littlewood also introduced the following double weighted inequality which was later generalized by Stein and Weiss in \cite{SW58}:
\begin{equation}\label{wHLSin}
    \ds\int_{\mathbb{R}^n}\int_{\mathbb{R}^n} \frac{f(x) g(y)}{|x|^{\alpha}|x - y|^{\lambda}|y|^{\beta}} dx dy \leq
    C_{\alpha,\beta,s,\lambda, n} \|f\|_r\|g\|_s
\end{equation} where $\alpha + \beta \geq 0$, $\alpha + \beta + \lambda \leq n$,
\begin{equation*} 
1 -\frac{1}{r} - \frac{\lambda}{n} < \frac{\alpha}{n} < 1 - \frac{1}{r}, \text{ and }~\frac{1}{r} + \frac{1}{s} + \frac{\lambda + \alpha + \beta}{n} = 2.
\end{equation*} 

The corresponding Euler--Lagrange equations for its associated functional is the system of integral equations
\begin{equation}
\left \{
\begin{array}{l}
      u(x) = \ds\frac{1}{|x|^{\alpha}}\int_{\mathbb{R}^{n}} \frac{v(y)^q}{|y|^{\beta}|x-y|^{\lambda}}  dy,  \,\,
      u > 0 \,\, \mbox{ in } \mathbb{R}^n,\\
      v(x) = \ds\frac{1}{|x|^{\beta}}\int_{\mathbb{R}^{n}} \frac{u(y)^p}{|y|^{\alpha}|x-y|^{\lambda}}dy, \,\,
      v > 0  \,\, \mbox{ in } \mathbb{R}^n,
\end{array}
\right. 
\end{equation}
where $0<p, q < \infty, \,\, 0< \lambda <n, \,\, \frac{\alpha}{n} < \frac{1}{p+1} < \frac{\lambda + \alpha}{n},$ and  $\frac{1}{1+p} + \frac{1}{1+q} = \frac{\lambda + \alpha + \beta}{n}$. In \cite{CL08}, Chen and Li examined this weighted HLS inequality and its corresponding Euler--Lagrange equations. As a result, the authors proved the uniqueness of solutions to the singular nonlinear system
\begin{equation}\label{singular}
  \left\{\begin{array}{cl}
    -\Delta(|x|^{\alpha} u) = \displaystyle\frac{v^{q}}{|x|^{\beta}} & 	\text{ in } \mathbb{R}^n\backslash\{0\}, \\
    -\Delta(|x|^{\beta} v)  = \displaystyle\frac{u^{p}}{|x|^{\alpha}} & 	\text{ in } \mathbb{R}^n\backslash\{0\},
  \end{array}
\right.
\end{equation}
and classified all the solutions for the case $\alpha = \beta$ and $p = q$, thereby obtaining the best constant in the corresponding weighted HLS inequality. Observe that if $f(x) = |x|^{\alpha}u(x)$ and $g(x) = |x|^{\beta}v(x)$, then \eqref{singular} becomes
\begin{equation*}
  \left\{\begin{array}{cl}
    -\Delta f(x) = \displaystyle\frac{g(x)^{q}}{|x|^{\beta(q+1)}}  & 	\text{ in } \mathbb{R}^n\backslash\{0\}, \\
    -\Delta g(x) = \displaystyle\frac{f(x)^{p}}{|x|^{\alpha(p+1)}} & 	\text{ in } \mathbb{R}^n\backslash\{0\},
  \end{array}
\right.
\end{equation*}
which is just a particular case of system \eqref{wHLS systemp}. Let us remark on the case of supercritical exponents for the HLS equations in relation to the work in this article. As a simple illustration, let $\gamma = 2k$, $u=v$ and $p=q$ in \eqref{HLS PDE} to obtain the scalar equation
\begin{equation}
      (-\Delta)^{k}u(x) = u(x)^{p},  \,\, 2k <n, \,\,
      u > 0 \,\, \mbox{ in } \mathbb{R}^n.
\label{PDEscalar} \end{equation}

In the supercritical case $p> \frac{n+2k}{n-2k}$ with $k=1$, the shooting method can be successfully applied to \eqref{PDEscalar}, however, much difficulty arises even in this scalar case with $k \geq 2$. In this paper, we circumvent these difficulties by further developing our degree theoretic framework for the shooting method to handle even more general systems such as weighted poly-harmonic systems, especially since existence results are not so well developed for these problems. As a result, we demonstrate how to handle even the case of \eqref{wHLS equationp} and \eqref{wHLS systemp}, which are not included in \cite{Li11} and \cite{LV12}. In addition to the non-existence results, the main obstacle in implementing our technique lies in determining the sufficient conditions for the continuity of the target map. This issue motivates our consideration of the {\bf non-degeneracy} conditions provided shortly below.

The rest of this manuscript is organized as follows. In section \ref{prelim}, we introduce some preliminary definitions and provide the precise statements of our main results. Section \ref{pothm1} gives the proof of our general existence theorem concerning the system \eqref{general}. In section \ref{section4}, we prove the existence theorems concerning equation \eqref{wHLS equationp} and system \eqref{wHLS systemp}. In view of Theorem \ref{theorem1}, to prove the existence of solutions for these weighted systems, some non-existence results for the corresponding boundary value problems are needed and whose proofs are also provided in Section \ref{section4}.

\section{Preliminaries and Main Results}\label{prelim}
Throughout this paper, we take $n\geq 3$, $x\in \mathbb{R}^{n}$, $\mathbb{R}_{+}:= [0,\infty)$ and $\mathbb{R}^{L}_{+}$ denotes the $L$-times Cartesian product of $\mathbb{R}_{+}$. For $v\in\mathbb{R}^{L}_{+}$, we say $v > 0$ if each component $v_{j} > 0$ for all $j=1,2,\ldots,L$.

Now consider the system 
\begin{align}\label{general1}
  \left\{\begin{array}{cl}
    (-\Delta)^{k_{i}}u_{i} = f_{i}(|x|,u_{1},u_{2},\ldots,u_{L}) & \text{ in } \mathbb{R}^n\backslash\{0\}, \\
    u_{i} > 0 & \text{ in } \mathbb{R}^n, \,\, \text{ for } i = 1,2,\ldots, L,
  \end{array}
\right.
\end{align}
with the following assumptions. We always assume that $k_i \geq 1$ and 
$$F(|x|,u) = (f_{1}(|x|,u), f_{2}(|x|,u),\ldots, f_{L}(|x|,u))$$
satisfies the following conditions:
\begin{enumerate}[(a)]
\item $F:(0,\infty) \times\mathbb{R}^{L}_{+} \longrightarrow \mathbb{R}^{L}_{+}$ is a continuous vector-valued map,
\item $F(|x|,u) > 0$ in the interior of $\mathbb{R}_{+} \times\mathbb{R}^{L}_{+}$,
\item $F$ is locally Lipschitz continuous in the second argument uniformly in the interior of $\mathbb{R}_{+} \times\mathbb{R}^{L}_{+}$.
\end{enumerate}

\noindent{\bf Non-degeneracy conditions:}
Let $F=F(|x|,u)$ satisfy the following:
\begin{enumerate}[(i)]
\item For each non-zero $v \in \partial \mathbb{R}^{L}_{+}$ there are constants $\lambda=\lambda(v)>0$, $\sigma=\sigma(v) > -2$ and a $\delta=\delta(v)>0$ such that if $|v - w| < \delta$, then 
\begin{equation*}
\lambda(v)|x|^{\sigma} \leq \displaystyle\sum_{j \in I_{v}^{0}} f_{j}(|x|,w) \, \text{ for } \, x \in \mathbb{R}^{n}\backslash \{0\},
\end{equation*}
where $I^{0}_{v}$ denotes the set of indices $j \in \{ 1,2,\ldots, L \}$ for which $v_j = 0$;
\item if $\ds\lim_{|x|\longrightarrow \infty} F(|x|,v) = 0$, then $v \in\partial \mathbb{R}^{L}_{+}$.

\end{enumerate}

In this paper, a system of the form \eqref{general} satisfying conditions (a)--(c) and the non-degeneracy conditions will simply be called a {\bf non-degenerate} system. 

\begin{remark}
It is not too difficult to check that for an even integer $\gamma = 2k$ with $k>1$, equation \eqref{wHLS equationp} and system \eqref{wHLS systemp}---after a reduction to a second-order system, if necessary---are examples of non-degenerate systems (cf. \eqref{wsys} in Section \ref{pothm1} for more details on this reduction). On the other hand, if $k=1$, then the non-degeneracy condition-(ii) may not hold due to the possible singular weights $|x|^{\sigma_i}$. Nevertheless, the existence result for \eqref{wHLS equationp} and \eqref{wHLS systemp} for $k=1$ still remains true (cf. section \ref{final section} for more details on circumventing this issue).
\end{remark}

The first theorem we present in this paper shows that the existence of solutions for non-degenerate systems follows from two parts. The first is the non-degeneracy conditions which allows us to apply a topological argument with the shooting method. The second part is the non-existence of solutions to the corresponding Navier boundary value problem
\begin{align}\label{general Navier}
  \left\{\begin{array}{cl}
    (-\Delta)^{k_{i}}u_{i} = f_{i}(|x|,u_{1},u_{2},\ldots,u_{L}) & \text{ in } B_{R}(0)\backslash\{0\}, \\
    u_{i} > 0 & \text{ in } B_{R}(0),  \\
    u_{i} = -\Delta u_{i} = \ldots = (-\Delta)^{k_{i}-1}u_{i} = 0 & \text{ on } \partial B_{R}(0), \,\, i = 1,2,\ldots, L,
  \end{array}
\right.
\end{align}
for all $R>0$. Here, $B_{R}(0)\subset \mathbb{R}^{n}$ denotes the open ball of radius $R$ centered at the origin with boundary $\partial B_{R}(0)$. 

\begin{theorem}\label{theorem1}
The non-degenerate system \eqref{general1} admits a radially symmetric solution of class $C^{2k}(\mathbb{R}^{n} \backslash\{0\})$ provided that \eqref{general Navier} admits no radially symmetric solution of class 

\noindent $C^{2k}(B_{R}(0)\backslash\{0\})\cap C^{2k-1}(\overline{B_{R}(0)})$ for all $R>0$. Furthermore, the solution is a ground-state solution i.e. it is bounded and satisfies the asymptotic property:
\begin{equation}\label{asymptotic}
u_{i} \longrightarrow 0 \,\text{ uniformly as }\, |x| \longrightarrow \infty \,\text{ for }\, i=1,2,\ldots,L.
\end{equation}
\end{theorem}
From Theorem \ref{theorem1}, it is clear that the following non-existence theorems will serve as important ingredients in proving existence results for \eqref{wHLS equationp} and \eqref{wHLS systemp}. Moreover, it may be interesting to note that these non-existence results readily hold on any bounded smooth domain star-shaped with respect to the origin.

\begin{theorem}\label{theorem2}
Let $k\in [1,n/2)$ be an integer, $p > 0$, and $\sigma \in (-\infty, n)$. Then the $2k$-th order equation
\begin{align}\label{Navier HLS equation}
  \left\{\begin{array}{ccc}
    (-\Delta)^{k} u = \displaystyle\frac{u^{p}}{|x|^{\sigma}} &\text{in }& B_{R}(0)\backslash \{0\}, \\
	u > 0 & \text{in }& B_{R}(0), \\
 	u=-\Delta u=\cdots=(-\Delta)^{k-1} u = 0 & \text{on }& \partial B_{R}(0),
        \end{array}
\right.
\end{align}
admits no radially symmetric solution of class $C^{2k}(B_{R}(0) \backslash \{0\})\cap C^{2k-1}(\overline{B_{R}(0)})$ for any $R>0$ provided that
\begin{equation}\label{scalar super-critical}
p\geq \frac{n+2k-2\sigma}{n-2k}.
\end{equation}
\end{theorem}

\begin{theorem}\label{theorem3}
Let $k\in [1,n/2)$ be an integer, $s,t \geq 0$, $p,q > 0$ and $\sigma_{1},\sigma_{2} \in (-\infty, n)$. Then the $2k$-th order system
\begin{align}\label{Navier HLS system}
  \left\{\begin{array}{ccc}
    (-\Delta)^{k} u = \displaystyle\frac{u^{s}v^{q}}{|x|^{\sigma_1}} &\text{in } B_{R}(0)\backslash \{0\}, \\
    (-\Delta)^{k} v = \displaystyle\frac{v^{t}u^{p}}{|x|^{\sigma_2}} &\text{in } B_{R}(0)\backslash \{0\}, \\
	u,~v > 0 & \text{in } B_{R}(0), \\
 	u=-\Delta u=\cdots=(-\Delta)^{k-1} u = 0 & \text{on } \partial B_{R}(0), \\
 	v=-\Delta v=\cdots=(-\Delta)^{k-1} v = 0 & \text{on } \partial B_{R}(0),
        \end{array}
\right.
\end{align}
admits no radially symmetric solution of class $C^{2k}(B_{R}(0) \backslash \{0\})\cap C^{2k-1}(\overline{B_{R}(0)})$ for any $R>0$ provided that
\begin{equation}\label{general super-critical}
\displaystyle \frac{n - \sigma_1 }{1 + q} + \frac{n - \sigma_2 }{1 + p} \leq n - 2k.
\end{equation}
\end{theorem}

Theorems \ref{theorem1} -- \ref{theorem3} have the following consequences.

\begin{corollary}\label{corollary1}
Let $k\in [1,n/2)$ be an integer, $p > 0$, and $\sigma \in (-\infty, 2)$. Then the $2k$-th order equation 
\begin{equation}\label{wHLS equation}
  \left\{\begin{array}{cl}
    (-\Delta)^{k} u = \displaystyle\frac{u^{p}}{|x|^{\sigma}} & \text{ in } \mathbb{R}^n\backslash\{0\}, \\
    u > 0 & \text{ in } \mathbb{R}^n,\\
    u \longrightarrow 0  \text{ uniformly as } & |x| \longrightarrow \infty,
  \end{array}
\right.
\end{equation}
admits a solution of class $C^{2k}(\mathbb{R}^{n}\backslash \{0\})$ provided that 
\begin{equation*}
p\geq \frac{n+2k- 2\sigma}{n-2k}.
\end{equation*}
\end{corollary}

\begin{corollary}\label{corollary2}
Let $k\in [1,n/2)$ be an integer, $p,q > 0$, and $\sigma_{1},\sigma_{2} \in (-\infty, 2)$. Then the $2k$-th order system 
\begin{equation}\label{wHLS system}
  \left\{\begin{array}{cl}
    (-\Delta)^{k} u = \displaystyle\frac{v^{q}}{|x|^{\sigma_1}} & 	\text{ in } \mathbb{R}^n\backslash\{0\}, \\
    (-\Delta)^{k} v = \displaystyle\frac{u^{p}}{|x|^{\sigma_2}} & 	\text{ in } \mathbb{R}^n\backslash\{0\}, \\
    u,~v > 0 & \text{ in } \mathbb{R}^n, \\
    u,v \longrightarrow 0  \text{ uniformly as } & |x| \longrightarrow \infty,
  \end{array}
\right.
\end{equation}
admits a solution of class $C^{2k}(\mathbb{R}^{n}\backslash \{0\})$ provided that 
\begin{equation*}
\frac{n -\sigma_1 }{1 + q} + \frac{n - \sigma_2}{1 + p} \leq n-2k.
\end{equation*}
\end{corollary}

\section{Proof of Theorem \ref{theorem1}}\label{pothm1}
In order to prove the first existence theorem, we must introduce several key ideas and lemmas. As mentioned earlier, the proof centers on a construction of a map which aims the shooting method. This section defines the target map and applies our method to prove Theorem \ref{theorem1}, but the proof on the continuity of the target map is provided later in Section \ref{section4}.

Before we can apply our method, we need to reduce the poly-harmonic system into a second-order system. For $i = 1,2,\ldots, L$ set $w_{i,j} = (-\Delta)^{j-1}u_{i}$, $1 \leq j \leq k_{i}$ so that
\begin{equation}\label{wsys}
\left\{ \begin{aligned}
        &-\Delta w_{i,1} = w_{i,2}, -\Delta w_{i,2} = w_{i,3}, \ldots, -\Delta w_{i,k_{i}-1} = w_{i,k_{i}},\\
        &-\Delta w_{i,k_{i}} = f_{i}(|x|,w_{1,1},w_{2,1},\ldots, w_{L,1}) & \text{ in }&\quad{\mathbb{R}^{n}\backslash\{0\}}, \\
        & w_{i,1}, w_{i,2},\ldots, w_{i,k_{i}}>0 &\text{ in }&\quad{\mathbb{R}^{n}},\\
        & \text{ where } i = 1,2,\ldots, L.     
        \end{aligned} \right.
\end{equation}
Solutions of \eqref{wsys} are clearly solutions of \eqref{general1}, so it will suffice to show the existence of solutions to \eqref{wsys} instead. The above system is an example of the more general system
\begin{equation}\label{w-reduced}
\left\{ \begin{aligned}
        &-\Delta w_1 = f_{1}(r,w),-\Delta w_{2} = f_{2}(r,w), \\
        & -\Delta w_{3}=f_{4}(r,w), \hdots, -\Delta w_{L-1}=f_{L-1}(r,w),\\
        &-\Delta w_{L}=f_{L}(r,w) &\text{ in } &\quad{\mathbb{R}^{n}\backslash\{0\}}, \\
        & w_{1}, w_{2},\ldots, w_{L}>0 &\text{ in } &\quad{\mathbb{R}^{n}},   
        \end{aligned} \right.
\end{equation}
where we are still using $L$ to represent the appropriate positive integer. It suffices to consider only \eqref{w-reduced} when proving both Theorem \ref{theorem1} and the continuity of the target map since the non-degeneracy conditions and similar arguments still hold even after this reduction to a second-order system. 

Let us define the aforementioned target map. For any strictly positive initial value $\alpha = (\alpha_1,\alpha_2,\ldots,\alpha_{L})$, consider the IVP
\begin{equation}\label{ODE}
\left\{ \begin{array}{cc}
w_{i}^{''}(r) + \dfrac{n-1}{r}w_{i}^{'}(r) = -f_{i}(r,w(r)), \\
w_{i}^{'}(0) = 0, \ w_{i}(0) = \alpha_i \text{ for } i = 1,2,\ldots, L.
\end{array}\right.
\end{equation}
\begin{definition}
Define the target map $\psi:\mathbb{R}^{L}_{+} \longrightarrow \mathbb{R}^{L}_{+}$ as follows. For $\alpha \in int(\mathbb{R}^{L}_{+})$, the interior of $\mathbb{R}^{L}_{+}$,
\begin{enumerate}[(a)]
\item $\psi(\alpha) = w(r_0)$ where $r_0$ is the smallest such $r$ for which $w_{i_0}(r) = 0$ for some $1\leq i_0 \leq L$,
\item otherwise, if no such $r_0$ exists, then $\psi(\alpha) = \lim_{r \rightarrow \infty} w(r)$.
\item $\psi \equiv Identity$ on the boundary $\partial \mathbb{R}^{L}_{+}$. 
\end{enumerate}
\end{definition}
\begin{remark}
We may think of (a) as the case when the solution hits the wall for the first time and (b) is the case where it never hits the wall. Observe also that $\psi$ is equivalent to the identity map on the wall. This property is crucial when we apply our topological degree argument for the shooting method. 
\end{remark}
\begin{remark}
Conditions (a)-(c) on pages 4 and 5 guarantee the existence of a unique solution for \eqref{ODE}. The goal here is to find the correct initial conditions which guarantee the positive solution of \eqref{ODE} never hits the wall and therefore global, thereby proving the desired existence result for the non-degenerate system \eqref{w-reduced}.
\end{remark}
The next lemma is a standard result from Brouwer topological degree theory and can be found in various sources (cf. \cite{AM07} and \cite{Nirenberg01} for instance).
\begin{lemma}\label{lemma1}
Let $U \subset \mathbb{R}^{n}$ be a bounded open set and $f,g:\overline{U}\longrightarrow \mathbb{R}^{n}$ are continuous maps. Suppose that $f \equiv g$ on $\partial U$ and $a \notin f(\partial U) = g(\partial U)$, then $degree(f,U,a) = degree(g,U,a)$.
\end{lemma}
One may recall the important property that if $degree(f,U,a) \neq 0$, then there exists a point $x \in U$ such that $f(x) = a$.
\begin{lemma}\label{lemma2}
The target map $\psi:\mathbb{R}^{L}_{+} \longrightarrow \partial\mathbb{R}^{L}_{+}$ is continuous.
\end{lemma}
We give the proof of this later in the final section.

\begin{lemma}\label{lemma3}
For every $a > 0$, there exists an $\alpha_{a} \in A_{a}$ where
\[ A_{a} := \left\{ \alpha \in \mathbb{R}^{L}_{+} \ \big | \ \sum_{i=1}^{L} \alpha_{i} = a \right \} \]
such that $\psi(\alpha_{a}) = 0$. 
\end{lemma}

\begin{proof}[Proof of Lemma \ref{lemma3}]
Define the set $B_{a}$ as follows
\begin{equation*}
B_{a} := \left\{ \alpha \in \partial\mathbb{R}^{L}_{+} \ \big | \ \sum_{i=1}^{L} \alpha_i \leq a \right\}.
\end{equation*}
It follows that $\psi$ maps $A_a$ into $B_a$ due to the non-increasing property of solutions. It suffices to show that $\psi:A_{a} \longrightarrow B_{a}$ is onto since then there exists an $\alpha_a \in A_a$ for which $\psi(\alpha_a) = 0$. Now define the continuous map $\phi:B_a \longrightarrow A_a$ by
\begin{equation*}
\phi(\alpha) = \alpha + \frac{1}{L}\left( a - \sum_{i}^{L} \alpha_i \right)(1,1,\cdots,1)
\end{equation*}
with continuous inverse $\phi^{-1}:A_a \longrightarrow B_a$ defined 
\begin{equation*}
\phi^{-1}(\alpha) = \alpha - \left( \min_{i=1,\cdots,L} \alpha_i \right)(1,1,\cdots,1).
\end{equation*}
Set $\eta = \phi \circ \psi:A_a \longrightarrow A_a$. Then $\eta$ is continuous on $A_a$ and is equivalent to the identity map on the boundary of $A_a$. By Lemma \ref{lemma1}, the index of the map satisfies $degree(\eta,A_a,\alpha) = degree(Identity,A_a,\alpha) = 1\neq 0$ for any interior point $\alpha \in int(A_a)$. So $\eta$ is onto, and thus $\psi$ is onto. 
\end{proof}

\noindent{\bf \underline{Proof of Theorem \ref{theorem1}}:} For fixed $a>0$, let $w = w(r)$ be the solution of \eqref{ODE} with initial condition $w(0) = \alpha_{a}$ as guaranteed by Lemma \ref{lemma3}. We claim this solution must never hit the wall. Otherwise, if this was the case, then there would be a smallest finite value $r = r_{0}$ such that $w(r_{0})= \psi(\alpha_{a}) = 0$. But this would imply that $w = w(|x|)$ is a radially symmetric solution of \eqref{general Navier} with $R = r_{0}$, which contradicts the non-existence assumption on all ball domains. Hence, the solution must never hit the wall, which implies that $w = w(|x|)$ is a radially symmetric solution of \eqref{general1}. Furthermore, the definition of the target map implies that $w \longrightarrow 0$ uniformly as $|x| \longrightarrow \infty$. \qed
\begin{remark}\label{remark2}
In the proof of Theorem \ref{theorem1}, we are using the fact that the Navier boundary value problem \eqref{general Navier} is equivalent to the reduced second-order system
\begin{equation}\label{Navier reduced}
\left\{ \begin{aligned}
        &-\Delta w_{i,1} = w_{i,2}, \\
        &-\Delta w_{i,2} = w_{i,3},\\
        & \hspace{0.3in}  \vdots\\
        &-\Delta w_{i,k_{i}-1} = w_{i,k_{i}},\\
        &-\Delta w_{i,k_{i}} = f_{i}(|x|,w_{1,1},w_{2,1},\ldots, w_{L,1}) & \text{ in }&\quad{ B_{R}(0)\backslash\{0\}}, \\
        & w_{i,1}, w_{i,2},\ldots, w_{i,k_{i}}>0 &\text{ in }&\quad{B_{R}(0)},\\
        & w_{i,1} = w_{i,2},\ldots, w_{i,k_{i}} =0 &\text{ on }& \quad{\partial B_{R}(0)},     
        \end{aligned} \right.
\end{equation}
where $w_{i,j} = (-\Delta)^{j-1}u_{i}$ for $j=1,2,\ldots,k_{i}$ and $i = 1,2,\ldots, L$. To see this equivalence, first observe that if $u_{1} = w_{1,1}, u_{2} = w_{2,1},\ldots, u_{L} = w_{L,1}$ where the $w_{i,j}\text{'s}$ satisfy \eqref{Navier reduced}, then $u_{1},u_{2},\ldots,u_{L}$ must satisfy \eqref{general Navier}. Conversely, suppose $u_1,u_2,\ldots,u_L$ satisfy \eqref{general Navier} and let 
\begin{equation}\label{sub}
w_{i,j} = (-\Delta)^{j-1}u_{i} \,\text{ for }\, j=1,2,\ldots,k_{i} \,\text{ and }\, i = 1,2,\ldots, L.
\end{equation} 
Notice that it is enough to show the super poly-harmonic property:
\begin{equation*}
w_{i,j} > 0 \,\text{ in }\, B_{R}(0) \,\text{ for  }\, j=1,2,\ldots,k_{i},\,\text{ and }\, i = 1,2,\ldots, L, 
\end{equation*}
since this would imply $w = (w_{i,j})$ under \eqref{sub} satisfies \eqref{Navier reduced}. Let us sketch the proof of this super poly-harmonic property. Since $w_{1,1},w_{2,1}$ and $w_{L,1}$ are positive in $B_{R}(0)$, we have that $-\Delta w_{i,k_{i}} > 0$ in $B_{R}(0)$ for $i=1,2,\ldots,L$. The boundary conditions along with the strong maximum principle imply that $w_{i,k_i} > 0$ in $B_{R}(0)$, which in turn, implies that $-\Delta w_{i,k_{i-1}} > 0$ in $B_{R}(0)$. Again, the boundary conditions and the strong maximum principle imply that $w_{i,k_{i-1}} > 0$ in $B_{R}(0)$ for $i=1,2,\ldots,L$. Obviously, we can repeat this argument successively to show the remaining components of $w=(w_{i,j})$ are positive in $B_{R}(0)$.
\end{remark}

\section{The Remaining Proofs}\label{section4}
This section first proves the non-existence theorems for the Navier boundary value problems. Then, the proof of Lemma \ref{lemma2} concerning the continuity of the target map is given followed by the proofs of Corollaries \ref{corollary1} and \ref{corollary2}.

\subsection{Non-existence of Solutions on Bounded Domains}
As demonstrated in Remark \ref{remark2} at the end of Section \ref{pothm1}, the proof of Theorem \ref{theorem2} reduces to showing the equivalent system,
\begin{equation}\label{wd-equation}
\left\{ \begin{aligned}
        & -\Delta w_{1} = w_{2}, -\Delta w_{2} = w_{3},\ldots, -\Delta w_{k-1}=w_{k}, \\
        & -\Delta w_{k} = \frac{w_{1}^{p}}{|x|^{\sigma}} & \text{ in }&\quad{B_{R}(0)\backslash \{0\}} \\
        & w_1, w_2,\ldots, w_{k}>0 &\text{ in }&\quad{B_{R}(0)}\\
        & w_1 = w_2 = \cdots = w_k = 0 &\text{ on }&\quad{\partial B_{R}(0)},                  
        \end{aligned} \right.
\end{equation}
admits no solution of class $C^{2}(B_{R}(0)\backslash \{0\})\cap C^{1}(\overline{B_{R}(0)})$ for any $R>0$. The key ingredients for this non-existence result center on establishing a Pohozaev type identity and the following identity.
\begin{lemma}\label{lemma4}
Let $w_{j}$ $(j = 1,2,\ldots,k)$ solve \eqref{wd-equation}. Then
\begin{align}
\int_{B_{R}(0)} \frac{u^{p+1}}{|x|^{\sigma}}\,dx = {} & \int_{B_{R}(0)} \nabla w_j \cdot \nabla w_{k+1-j}\,dx \notag \\
= {} & \int_{B_{R}(0)} w_{j+1} w_{k+1-j} \,dx =: E_{1},
\end{align}
(Here, it should be understood that $w_{k+1} := u^p/ |x|^{\sigma} = -\Delta w_{k}$).
\end{lemma}

\begin{proof}
To prove this lemma, multiply the $k$-th equation in (\ref{wd-equation}) by $w_1$ then integrate over $B_{R}(0)$. The repeated application of integration by parts with the boundary conditions imply
\begin{align*}
\int_{B_{R}(0)} {} & \frac{w_{1}^{p+1}}{|x|^{\sigma}}\,dx = \int_{B_{R}(0)} \nabla w_{1} \cdot \nabla w_{k}\,dx \\
= {} & -\int_{B_{R}(0)} w_{k} \Delta w_{1}\,dx = \int_{B_{R}(0)} w_{k}w_{2} \\
= {} & -\int_{B_{R}(0)} w_{2} \Delta w_{k-1}\,dx = \int_{B_{R}(0)}\nabla w_{2}\cdot \nabla w_{k-1}\,dx \\
= {} & -\int_{B_{R}(0)} w_{k-1}\Delta w_{2}\,dx = \int_{B_{R}(0)} w_{k-1}w_{3} \,dx \\
\vdots \\
= {} & \int_{B_{R}(0)} \nabla w_j \cdot \nabla w_{k+1-j}\,dx = \int_{B_{R}(0)} w_{j+1} w_{k+1-j} \,dx.
\end{align*}
\end{proof}

\begin{remark}
Let us be more precise in the calculations found in our proof of Lemma \ref{lemma4} since we employ similar calculations below. For instance, when we multiply, say, the $k$-th equation $-\Delta w_{k} = |x|^{-\sigma}w_{1}^{p}$ by $w_1$ then integrate over the ball $B_{R}(0)$, this should be understood in the following way: We integrate over $B_{R}(0)\backslash B_{\epsilon}(0)$ for $0 < \epsilon < R$ and use an integration by parts to obtain
\begin{align*}
\int_{B_{R}(0)\backslash B_{\epsilon}(0)} \frac{w_{1}^{p+1}}{|x|^{\sigma}} \,dx 
= {} & -\int_{B_{R}(0)\backslash B_{\epsilon}(0)} w_{1} \Delta w_{k} \,dx \\
= {} & -\int_{\partial B_{\epsilon}(0)} w_{1} \frac{\partial w_{k}}{\partial \nu} \,ds + \int_{B_{R}(0)\backslash B_{\epsilon}(0)} \nabla w_{1} \cdot \nabla w_{k} \,dx,
\end{align*}
where $\nu$ is the inward unit normal vector along $\partial B_{\epsilon}(0)$. By taking the limit as $\epsilon$ tends to zero, we obtain
\begin{equation*}
\int_{B_{R}(0)} \frac{w_{1}^{p+1}}{|x|^{\sigma}} \,dx = \int_{B_{R}(0)} \nabla w_{1} \cdot \nabla w_{k} \,dx.
\end{equation*}
All such calculations including those found in the proof of Theorems \ref{theorem2} and \ref{theorem3} below should be understood in this way.
\end{remark}

\begin{proof}[{\bf \underline{Proof of Theorem \ref{theorem2}}}]
The proof is by contradiction. Assume $w$ is a solution of \eqref{wd-equation}. For $j = 1,2,3,\ldots,k$, multiply the $j$-th equation in \eqref{wd-equation} by $x \cdot \nabla w_{k+1-j}$, integrate over $B_{R}(0)$, then integrate by parts to obtain
\begin{align}\label{identity 10}
-\int_{\partial B_{R}(0)} \frac{\partial w_j}{\partial n}\frac{\partial w_{k+1-j}}{\partial n}(x\cdot n)\,ds {} & + \int_{B_{R}(0)} \nabla w_j \cdot \nabla w_{k+1-j}\,dx + \int_{B_{R}(0)} x\cdot \nabla (w_{k+1-j})_{x_i}w_{j,x_{i}}\,dx \notag \\
= {} & \int_{B_{R}(0)}w_{j+1}(x\cdot \nabla w_{k+1-j}) + w_{k+2-j}(x\cdot \nabla w_j) \,dx,
\end{align}
where $n$ is the outward pointing unit normal vector. Similarly, multiply the $(k+1-j)$-th equation in \eqref{wd-equation} by $x \cdot \nabla w_{j}$, integrate over $B_{R}(0)$, then integrate by parts to obtain
\begin{align}\label{identity 11}
-\int_{\partial B_{R}(0)} \frac{\partial w_j}{\partial n}\frac{\partial w_{k+1-j}}{\partial n}(x\cdot n)\,ds {} & + \int_{B_{R}(0)} \nabla w_j \cdot \nabla w_{k+1-j}\,dx + \int_{B_{R}(0)} x\cdot \nabla (w_{j})_{x_i}w_{k+1-j,x_{i}}\,dx \notag \\
= {} & \int_{B_{R}(0)}w_{j+1}(x\cdot \nabla w_{k+1-j}) + w_{k+2-j}(x\cdot \nabla w_j) \,dx. 
\end{align}
By adding \eqref{identity 10} and \eqref{identity 11} together and using integration by parts, we obtain
\begin{align}\label{identity 12}
-\int_{\partial B_{R}(0)} \frac{\partial w_j}{\partial n}\frac{\partial w_{k+1-j}}{\partial n}(x\cdot n)\,ds {} & + (2-n)\int_{B_{R}(0)} \nabla w_j \cdot \nabla w_{k+1-j}\,dx \notag \\
= {} & \int_{B_{R}(0)}w_{j+1}(x\cdot \nabla w_{k+1-j}) + w_{k+2-j}(x\cdot \nabla w_j) \,dx. 
\end{align}
In addition, observe that integration by parts and the boundary conditions yield the identity
\begin{align*}
\int_{B_{R}(0)} x \cdot (w_{j}\nabla w_{k+2-j} + w_{k+2-j}\nabla w_{j})\,dx = {} & \int_{B_{R}(0)} x\cdot \nabla (w_{j}w_{k+2-j})\,dx \\
= {} & -n\int_{B_{R}(0)} w_{j} w_{k+2-j}\,dx.
\end{align*}
With this identity, summing (\ref{identity 12}) over $j = 1,2,3,\ldots,k$ gives us the Pohozaev type identity 
\begin{align*}
(2-n)\displaystyle \sum_{j=1}^{k}\int_{B_{R}(0)} \nabla w_{j} \cdot \nabla w_{k+1-j}\,dx + {} & \displaystyle \sum_{j=1}^{k-1} n\int_{B_{R}(0)}w_{j+1}w_{k+1-j}\,dx +  \frac{2(n - \sigma)}{1+p}\int_{B_{R}(0)} \frac{w_{1}^{p+1}}{|x|^{\sigma}}\,dx \\
= \displaystyle \sum_{j=1}^{k} \int_{\partial B_{R}(0)} {} & \frac{\partial w_j}{\partial n}\frac{\partial w_{k+1-j}}{\partial n}(x\cdot n)\,ds.
\end{align*}
Observe that the right hand side of this identity must be strictly positive by the non-increasing property of the positive radial solutions. Hence, Lemma \ref{lemma4} implies that
\begin{equation*}
\Big\lbrace k(2-n) + n(k-1) + \frac{2(n - \sigma)}{1 + p} \Big\rbrace \cdot E_{1} > 0,
\end{equation*}
and we arrive at a contradiction.
\end{proof}
Similarly, the key ingredients in the proof of Theorem \ref{theorem3} is a Pohozaev type identity and the following identity.
\begin{lemma}\label{lemma5}
Let $w_{j}$ $(j = 1,2,\ldots,2k)$ solve
\begin{equation}\label{wd-system}
\left\{ \begin{aligned}
        &-\Delta w_1 = w_2, 
        \hdots, -\Delta w_{k-1}=w_k,\\
        & -\Delta w_{k} = \frac{w_{1}^{s}w_{k+1}^{q}}{|x|^{\sigma_1}},\\
        & -\Delta w_{k+1}=w_{k+2}, \hdots, -\Delta w_{2k-1}=w_{2k},\\
        &-\Delta w_{2k} = \frac{w_{k+1}^{t}w_{1}^{p}}{|x|^{\sigma_2}} & \text{ in }&\quad{B_{R}(0)\backslash \{0\}} \\
        & w_1, w_2,\ldots, w_{2k}>0 &\text{ in }&\quad{B_{R}(0)}\\
        & w_1 = w_{2} = \cdots = w_{2k}=0 &\text{ on }&\quad{\partial B_{R}(0)}.                  
        \end{aligned} \right.
\end{equation}
Then
\begin{align}
\int_{B_{R}(0)} \frac{u^{s}v^{q+1}}{|x|^{\sigma_1}}\,dx = {} & \int_{B_{R}(0)} \frac{v^{t}u^{p+1}}{|x|^{\sigma_2}}\,dx \notag \\
= {} & \int_{B_{R}(0)} \nabla w_j \cdot \nabla w_{2k+1-j}\,dx \notag \\
= {} & \int_{B_{R}(0)} w_{j+1} w_{2k+1-j} \,dx =: E_{2}. \label{E}
\end{align}
(Here, it should be understood that $w_{2k+1} := v^{t}u^{p}/ |x|^{\sigma_{2}} = -\Delta w_{2k}$).
\end{lemma}
\begin{proof}
To prove this lemma, multiply the $2k$-th equation in (\ref{wd-system}) by $w_1$ then integrate over $B_{R}(0)$. The repeated application of integration by parts along with the boundary conditions yield
\begin{align*}
\int_{B_{R}(0)} {} & \frac{w_{k+1}^{t}w_{1}^{p+1}}{|x|^{\sigma_2}}\,dx = \int_{B_{R}(0)} \nabla w_{1} \cdot \nabla w_{2k}\,dx \\
= {} & -\int_{B_{R}(0)} w_{2k} \Delta w_{1}\,dx = \int_{B_{R}(0)} w_{2k}w_{2} \\
= {} & -\int_{B_{R}(0)} w_{2} \Delta w_{2k-1}\,dx = \int_{B_{R}(0)}\nabla w_{2}\cdot \nabla w_{2k-1}\,dx \\
= {} & -\int_{B_{R}(0)} w_{2k-1}\Delta w_{2}\,dx = \int_{B_{R}(0)} w_{2k-1}w_{3} \,dx \\
\vdots \\
= {} & \int_{B_{R}(0)} \nabla w_{k} \cdot \nabla w_{k+1}\,dx = -\int_{B_{R}(0)} w_{k+1}\Delta w_{k}\,dx  \\
= {} & \int_{B_{R}(0)} \frac{w_{1}^{s}w_{k+1}^{q+1}}{|x|^{\sigma_1}}\,dx.
\end{align*}
\end{proof}

\begin{proof}[{\bf \underline{Proof of Theorem \ref{theorem3}}}]
Assume that $w = (w_{j})$ is a solution of \eqref{wd-system} with non-negative exponents satisfying \eqref{general super-critical}. For $j = 2,3,\ldots,k-1$, multiply the $j$-th equation in \eqref{wd-system} by $x \cdot \nabla w_{2k+1-j}$, integrate over $B_{R}(0)$, then integrate by parts to obtain
\begin{align}\label{identity 6}
-\int_{\partial B_{R}(0)} {} & \frac{\partial w_j}{\partial n}\frac{\partial w_{2k+1-j}}{\partial n}(x\cdot n)\,ds + \int_{B_{R}(0)} \nabla w_j \cdot \nabla w_{2k+1-j}\,dx \\
 ~+ {} & \int_{B_{R}(0)}x\cdot w_{j,x_i}\nabla (w_{2k+1-j})_{x_i}\,dx = \int_{B_{R}(0)} w_{j+1}(x\cdot \nabla w_{2k+1-j})\,dx. \notag
\end{align}
Multiply the $(2k+1-j)$--th equation in \eqref{wd-system} by $x\cdot \nabla w_j$ and integrate over $B_{R}(0)$ and perform analogous calculations as was done in obtaining \eqref{identity 6}. Then summing the resulting equation with \eqref{identity 6} and using the identity,
\begin{align*}
\int_{B_{R}(0)} {} & x\cdot w_{j,x_{i}}\nabla (w_{2k+1-j})_{x_i} + x\cdot w_{2k+1-j,x_{i}}\nabla (w_{j})_{x_i}\,dx \\
= {} & \int_{B_{R}(0)} x\cdot \nabla(\nabla w_j \cdot \nabla w_{2k+1-j})\,dx \\
= {} & -n\int_{B_{R}(0)} \nabla w_j \cdot \nabla w_{2k+1-j}\,dx + \int_{\partial B_{R}(0)} \frac{\partial w_j}{\partial n}\frac{\partial w_{2k+1-j}}{\partial n}(x\cdot n)\,ds,
\end{align*}
we obtain
\begin{align}\label{identity 7}
-\int_{\partial B_{R}(0)} {} & \frac{\partial w_j}{\partial n}\frac{\partial w_{2k+1-j}}{\partial n}(x\cdot n)\,ds + (2-n)\int_{B_{R}(0)} \nabla w_j \cdot \nabla w_{2k+1-j}\,dx \\
= {} & \int_{B_{R}(0)}w_{j+1}(x\cdot \nabla w_{2k+1-j}) + w_{2k+2-j}(x\cdot \nabla w_j) \,dx. \notag
\end{align}
%-------------------------------------------------------------------------------------
Now multiply the $2k$--th equation in (\ref{wd-system}) by $x\cdot \nabla w_{1}$ and integrate over $B_{R}(0)$ to obtain
\begin{equation*}
\underbrace{-\int_{B_{R}(0)} (x\cdot \nabla w_{1})\Delta w_{2k} \,dx}_{:=I_1} = \underbrace{\int_{B_{R}(0)} (x\cdot \nabla w_{1}) \frac{w_{k+1}^{t}w_{1}^{p}}{|x|^{\sigma_2}}\,dx}_{:=I_2}.
\end{equation*}
Let us calculate $I_1$ and $I_2$. Using integration by parts, 
\begin{equation*}
I_1 = -\int_{\partial B_{R}(0)} \frac{\partial w_1}{\partial n}\frac{\partial w_{2k}}{\partial n}(x\cdot n) \,ds + \int_{B_{R}(0)} \nabla w_1 \cdot \nabla w_{2k}\,dx + \int_{B_{R}(0)} x_{i}\frac{\partial w_{2k}}{\partial x_j}\left( \frac{\partial^2 w_1}{\partial x_j \partial x_i} \right)\,dx,
\end{equation*}
and
\begin{align*}
I_2 = {} & \frac{1}{1 + p}\int_{B_{R}(0)} x_{i} \frac{w_{k+1}^{t}\left( w_{1}^{p+1} \right)_{x_i}}{|x|^{\sigma_2}} \,dx \\
= {} & -\frac{n - \sigma_2}{1 + p}\int_{B_{R}(0)} \frac{w_{k+1}^{t}w_{1}^{p+1}}{|x|^{\sigma_2}}\,dx - \frac{t}{1+p}\int_{B_{R}(0)} \frac{w_{k+1}^{t-1}w_{1}^{p+1}}{|x|^{\sigma_2}}(x\cdot \nabla w_{k+1}) \,dx.
\end{align*}
Now multiply the first equation by $x\cdot \nabla w_{2k}$ and integrate over $B_{R}(0)$ to obtain
\begin{equation*}
\underbrace{-\int_{B_{R}(0)} (x\cdot \nabla w_{2k})\Delta w_{1} \,dx}_{:=II_1} = \underbrace{\int_{B_{R}(0)} (x\cdot \nabla w_{2k}) w_{2}\,dx}_{:=II_2}.
\end{equation*}
We use integration by parts to rewrite $II_1$ as follows.
\begin{equation*}
II_1 = -\int_{\partial B_{R}(0)} \frac{\partial w_{2k}}{\partial n}\frac{\partial w_1}{\partial n} (x\cdot n)\,ds + \int_{B_{R}(0)} \nabla w_1 \cdot \nabla w_{2k} \,dx + \int_{B_{R}(0)} x_{i}\frac{\partial w_{1}}{\partial x_j}\left( \frac{\partial^2 w_{2k}}{\partial x_j \partial x_i} \right)\,dx.
\end{equation*}
By summing together the two equations $I_1 = I_2$ and $II_1 = II_2$ and using the fact that
\begin{equation*}
\int_{B_{R}(0)} x \cdot \nabla (\nabla w_1 \cdot \nabla w_{2k})\,dx = \int_{\partial B_{R}(0)} \frac{\partial w_{2k}}{\partial n}\frac{\partial w_1}{\partial n} (x\cdot n)\,ds - n\int_{B_{R}(0)} \nabla w_1 \cdot \nabla w_{2k} \,dx,
\end{equation*}
we obtain the identity
\begin{align}\label{identity 8}
(2-n)\int_{B_{R}(0)} \nabla w_{2k} \cdot \nabla w_{1}\,dx {} & + \frac{n - \sigma_2 }{1+p}\int_{B_{R}(0)} \frac{w_{k+1}^{t}w_{1}^{p+1}}{|x|^{\sigma_2 }}\,dx = \int_{\partial B_{R}(0)} \frac{\partial w_{2k}}{\partial n}\frac{\partial w_1}{\partial n}(x\cdot n)\,ds \\
+ {} & \int_{B_{R}(0)}w_{2}(x\cdot \nabla w_{2k})\,dx - \frac{t}{1 + p}\int_{B_{R}(0)} \frac{w_{1}^{p+1}w_{k+1}^{t-1}}{|x|^{\sigma_2 }}(x\cdot \nabla w_{k+1})\,dx. \notag
\end{align}
Multiply the $k$--th and $(k+1)$--th equations in \eqref{wd-system} by $x\cdot \nabla w_{k+1}$ and $x\cdot \nabla w_{k}$, respectively, and integrate over $B_{R}(0)$. Using similar calculations to those used in deriving \eqref{identity 8}, we obtain
\begin{align}\label{identity 9}
(2-n)\int_{B_{R}(0)} \nabla w_{k} \cdot \nabla w_{k+1}\,dx {} & + \frac{n - \sigma_1 }{1+q}\int_{B_{R}(0)} \frac{w_{1}^{s}w_{k+1}^{q+1}}{|x|^{\sigma_1}}\,dx = \int_{\partial B_{R}(0)} \frac{\partial w_{k}}{\partial n}\frac{\partial w_{k+1}}{\partial n}(x\cdot n)\,ds \\
+ {} & \int_{B_{R}(0)}w_{k+2}(x\cdot \nabla w_{k})\,dx - \frac{s}{1 + q}\int_{B_{R}(0)} \frac{w_{k+1}^{q+1}w_{1}^{s-1}}{|x|^{\sigma_1 }}(x\cdot \nabla w_{1})\,dx. \notag
\end{align}
%----------------------------------------------------------------------------------------------------
Observe also that integrating by parts and using the boundary conditions, we obtain
\begin{align*}
\int_{B_{R}(0)} x \cdot (w_{j+1}\nabla w_{2k+1-j} + w_{2k+1-j}\nabla w_{j+1})\,dx = {} & \int_{B_{R}(0)} x\cdot \nabla (w_{j+1}w_{2k+1-j})\,dx \\
= {} & -n\int_{B_{R}(0)} w_{j+1} w_{2k+1-j}\,dx.
\end{align*}
Using this identity and summing (\ref{identity 7}) over $j = 2,3,\ldots,k-1$ along with (\ref{identity 8}) and (\ref{identity 9}), we arrive at the following Pohozaev type identity:
\begin{align*}
(2-n)\displaystyle \sum_{j=1}^{k}\int_{B_{R}(0)} {} &  \nabla w_{j} \cdot \nabla w_{2k+1-j}\,dx + \frac{n - \sigma_1 }{1+q}\int_{B_{R}(0)} \frac{w_{1}^{s}w_{k+1}^{q+1}}{|x|^{\sigma_1}}\,dx \\ 
{} & + \frac{n - \sigma_2 }{1+p}\int_{B_{R}(0)} \frac{w_{k+1}^{t}w_{1}^{p+1}}{|x|^{\sigma_2}}\,dx + \displaystyle \sum_{j=1}^{k-1} n\int_{B_{R}(0)}w_{j+1}w_{2k+1-j}\,dx \\
 = \displaystyle \sum_{j=1}^{k} \int_{\partial B_{R}(0)} {} & \frac{\partial w_j}{\partial n}\frac{\partial w_{2k+1-j}}{\partial n}(x\cdot n)\,ds - \Big\lbrace \frac{s}{1 + q}\int_{B_{R}(0)} \frac{w_{k+1}^{q+1}w_{1}^{s-1}}{|x|^{\sigma_1}}(x\cdot \nabla w_{1})\,dx \\
+ {} & \frac{t}{1 + p}\int_{B_{R}(0)} \frac{w_{1}^{p+1}w_{k+1}^{t-1}}{|x|^{\sigma_2}}(x\cdot \nabla w_{k+1})\,dx \Big\rbrace.
\end{align*}
Observe that the right hand side of this Pohozaev type identity must be strictly positive by the non-increasing property of the positive radial solutions. Hence, Lemma \ref{lemma5} implies that
\begin{equation*}
\Big\lbrace k(2-n) + \frac{n - \sigma_1 }{1 + q} + \frac{n - \sigma_2}{1 + p} + (k-1)n \Big\rbrace \cdot E_{2} > 0.
\end{equation*}
Thus, we have 
\begin{equation*}
\displaystyle\frac{n - \sigma_1}{1 + q} + \frac{n - \sigma_2}{1 + p} > n - 2k,
\end{equation*}
but this contradicts with (\ref{general super-critical}).
\end{proof}

\subsection{Continuity of the Target Map}

\begin{proof}[Proof of Lemma \ref{lemma2}]
Fix an $\epsilon > 0$ and choose any $\overline{\alpha} \in \mathbb{R}^{L}_{+}$. To show the continuity of the target map at $\overline{\alpha}$, there are three cases to consider:
\begin{enumerate}[\text{Case} (1):]
\item $\overline{\alpha} \in \partial \mathbb{R}^{L}_{+}$,
\item $\overline{\alpha} \in int(\mathbb{R}^{L}_{+})$ and the solution of \eqref{ODE} with initial condition $\overline{\alpha}$ hits the wall,
\item $\overline{\alpha} \in int(\mathbb{R}^{L}_{+})$ and the solution of \eqref{ODE} with initial condition $\overline{\alpha}$ never hits the wall.
\end{enumerate}

\noindent{\bf Case (1):} By definition, $\psi(\overline{\alpha}) = \overline{\alpha}$ for this case. If $\overline{\alpha} = 0$, the continuity of $\psi$ at $\overline{\alpha}$ follows easily from the non-increasing property of solutions since $|\psi(\alpha) - \psi(\overline{\alpha})|= |\psi(\alpha)| \leq \alpha \longrightarrow 0$ as $\alpha \longrightarrow \overline{\alpha}$.

So, we assume $\overline{\alpha}\in \partial \mathbb{R}^{L}_{+}$ is a non-zero boundary point. From the non-degeneracy conditions, we can find a $\delta_{1} > 0$ with $\delta_{1} < \epsilon$ such that for $|\alpha - \overline{\alpha}| < \delta_1 = \delta_{1}(\overline{\alpha})$,
\begin{equation*}
\displaystyle\sum_{j \in I^{0}_{\overline{\alpha}}} f_{j}(r,\alpha) \geq \lambda(\overline{\alpha})r^{\sigma} \,\, \text{ for } r > 0.
\end{equation*}
Then from basic ODE theory, we can find $\delta_2 > 0$ such that $|\overline{\alpha} - w(r,\alpha)| < \delta_1$ for $r < \delta_2$ and $|\alpha - \overline{\alpha}|< \delta_2 $ before the solution hits the wall. If we set
\[ W_{0}(r,\alpha) := \displaystyle\sum_{j\in I^{0}_{\overline{\alpha}}} w_{j}(r,\alpha), \] 
the non-degeneracy condition-(i) and \eqref{ODE} imply that
\[ -\frac{d}{dr}\left(r^{n-1}\frac{d}{dr}W_{0}(r,\alpha)\right) \geq \lambda(\overline{\alpha})r^{n-1 + \sigma}. \]
Integrating this twice with respect to $r$ yields
\[  W_{0}(r,\alpha) \leq \left( \displaystyle\sum_{j\in I^{0}_{\overline{\alpha}}} \alpha_j\right) - \frac{\lambda(\overline{\alpha})}{(2+ \sigma)(n+\sigma)} r^{2 + \sigma},\]
thus
\[  w_{j}(r,\alpha) \leq W_{0}(r,\alpha) \leq \left( \displaystyle\sum_{j\in I^{0}_{\overline{\alpha}}} \alpha_j\right) - \frac{\lambda(\overline{\alpha})}{(2+ \sigma)(n+\sigma)} r^{2 + \sigma} \,\text{ for }\, j \in I^{0}_{\overline{\alpha}}.\]
We can then choose $\delta > 0$ sufficiently small so that if $|\alpha - \overline{\alpha}| < \delta$, then there is a smallest value $r_{\alpha} \leq \delta_{2}$ for which $w_{j_0}(r_{\alpha},\alpha) = 0$ for some $j_{0} \in  I^{0}_{\overline{\alpha}}$. Hence,
\begin{equation*}
|\psi(\alpha) - \psi(\overline{\alpha})| = |\psi(\alpha) - \overline{\alpha}| \leq |w(r_{\alpha},\alpha) - \overline{\alpha}| < \epsilon \,\text{ whenever }\, |\alpha - \overline{\alpha}| < \delta.
\end{equation*}

\noindent{\bf Case (2):} Since the source terms $f_{i}$ are non-negative, we have that $u_{i_{0}}^{'}(r_0, \overline{\alpha}) < 0$ by a direct computation or simply by Hopf's Lemma. This transversality condition along with the ODE stability imply that for $\alpha$ sufficiently close to $\overline{\alpha}$, the solution to this perturbed IVP must hit the wall and $\psi(\alpha)$ must be close to $\psi(\overline{\alpha})$. \\

\noindent{\bf Case (3):} Observe that from \eqref{ODE}, we have that $0=\lim_{r\longrightarrow \infty} F(r,\psi(\overline{\alpha}))$ by elementary ODE or elliptic theory. The non-degeneracy condition-(ii) further implies that $\psi(\overline{\alpha}) \in \partial \mathbb{R}^{L}_{+}$. In fact, we claim that $\psi(\overline{\alpha}) = 0$. To see this, assume otherwise i.e. assume $\psi(\overline{\alpha})$ is a non-zero boundary point. Without loss of generality, we can assume from \eqref{ODE} and the non-degeneracy condition-(i) that there is a $j_0$ such that
\begin{equation*}
-\frac{d}{dr}\left(r^{n-1}\frac{d}{dr}w_{j_0}(r,\overline{\alpha})\right) = r^{n-1}f_{j_0}(r,w(r,\overline{\alpha})) \geq \lambda(\psi(\overline{\alpha})) r^{n-1+\sigma} \,\text{ for }\, r \geq r_1 \gg 1.
\end{equation*}
Here, $r_1$, $\lambda$ and $\sigma$ are suitable constants depending on $\psi(\overline{\alpha})$ and we are choosing $r_1 \gg 1$ so that $w(r,\overline{\alpha})$ is sufficiently close to $\psi(\overline{\alpha})$. From this, we have that 
\begin{equation*}
w_{j_0}(r,\overline{\alpha}) \leq C - \frac{\lambda}{(n+\sigma)(2+\sigma)}r^{2+\sigma} \,\text{ for }\, r \geq r_1,
\end{equation*}
where $C=C(n,r_1,\lambda,\sigma)>0$ is some constant. But this implies that the solution must hit the wall, which contradicts that $w(r,\overline{\alpha})$ is a positive entire solution of \eqref{ODE}. Thus, $\psi(\overline{\alpha})=0$.

Since $\psi(\overline{\alpha}) = 0$, we can choose a sufficiently large $R>0$ so that $|u(R,\overline{\alpha})| < \epsilon/2$. Then, by ODE stability, we can choose a $\delta > 0$ for which
\begin{equation*}
u(r,\alpha) > 0 \,\,\text{ and }\,\, |u(r,\alpha) - u(r,\overline{\alpha})| < \epsilon/2 \,\,\text{ on }\,\, [0,R] \,\,\text{ whenever } |\overline{\alpha} - \alpha| < \delta.
\end{equation*}
Hence,
\begin{equation*}
|\psi(\overline{\alpha}) - \psi(\alpha)| = |\psi(\alpha)| \leq |u(R,\alpha)| \leq |u(R,\overline{\alpha}) - u(R,\alpha)| + |u(R,\overline{\alpha})| <\epsilon. 
\end{equation*}
This completes the proof that $\psi$ is continuous at $\overline{\alpha}\in \mathbb{R}^{L}_{+}$.
\end{proof}

\subsection{Proof of Corollaries \ref{corollary1} and \ref{corollary2}}\label{final section} It is straightforward to check that equation \eqref{wHLS equation} and system \eqref{wHLS system}---after a reduction to a second-order system, if necessary---are non-degenerate for $k > 1$. Subsequently, Corollary \ref{corollary1} is a direct consequence of Theorems \ref{theorem1} and \ref{theorem2} and Corollary \ref{corollary2} is a consequence of Theorems \ref{theorem1} and \ref{theorem3}. However, the non-degeneracy condition-(ii) does not necessarily hold when $k=1$. Nevertheless, it suffices to show that the continuity of the target map still holds for this case since the existence result of both corollaries will follow accordingly. Let us describe how to show the continuity of the target map when $k=1$. By adopting similar arguments used in the proof of Lemma \ref{lemma2} for Case (iii), we can still show that the entire positive solutions must decay to zero at infinity. Thus, we arrive at the same conclusion that $\psi(\overline{\alpha}) = 0$ and the continuity of the target map still holds for this case. The continuity of $\psi$ at $\overline{\alpha}$ under Case (i) and (ii) also holds and follows the same exact arguments used in the proof of Lemma \ref{lemma2}. This completes the proof of both corollaries. \qed \\

\noindent{\bf Acknowledgement}
\medskip

The author would like to express his appreciation to Professor Congming Li for his guidance and for the meaningful discussions.

%\bibliographystyle{plain}
%\bibliography{refs}

\bibliographystyle{elsarticle-num.bst}
%\bibliography{<refs.bib>}

\end{document}